\documentclass[11pt]{article}
\usepackage[a4paper, total={5.5in, 9in}]{geometry}
\usepackage{amsmath, amsfonts,amsthm, amssymb}
\usepackage{tikz}
\usepackage{todonotes}
\usepackage{multicol}
\usepackage[pagebackref,colorlinks=true,citecolor=blue]{hyperref}

\newtheorem{theorem}{Theorem}
\newtheorem{lemma}[theorem]{Lemma}

\newtheorem{corollary}[theorem]{Corollary}

\DeclareMathOperator{\Cay}{Cay}
\DeclareMathOperator{\Aut}{Aut}
\DeclareMathOperator{\Inn}{Inn}

\newcommand*{\id}{\mathrm{id}}

\title{Automorphism groups of Cayley graphs generated by general transposition sets}
\author{Dion Gijswijt \thanks{Delft Institute of Applied Mathematics, Delft University of Technology, The Netherlands, {\tt d.c.gijswijt@tudelft.nl}} \and Frank de Meijer \thanks{Delft Institute of Applied Mathematics, Delft University of Technology, The Netherlands, {\tt f.j.j.demeijer@tudelft.nl}}}
\date{\today}

\begin{document}
\maketitle
\begin{abstract}
    In this paper we study the Cayley graph $\Cay(S_n,T)$ of the symmetric group $S_n$ generated by a set of transpositions $T$. We show that for $n\geq 5$ the Cayley graph is normal. As a corollary, we show that its automorphism group is a direct product of $S_n$ and the automorphism group of the transposition graph associated to $T$. This provides an affirmative answer to a conjecture raised by A.~Ganesan, \textit{Cayley graphs and symmetric interconnection networks}, showing that $\Cay(S_n,T)$ is normal if and only if the transposition graph is not $C_4$ or~$K_n$. 
\end{abstract}


\section{Introduction}
Given a finite group $H$ and a generating subset $T \subseteq H$ with $T = T^{-1}$ and $\id \notin T$, the \emph{Cayley graph of $H$ with respect to $T$} is the simple, undirected, connected graph defined as
\[
\Cay(H,T) :=(H, \{a,ta\}:a\in H, t\in T\}).
\]

The Cayley graph is vertex transitive as its automorphism group $\Aut(\Cay(H,T))$ contains the \emph{right regular representation} $R(H)=\{\rho_a: a\in H\}$, where $\rho_a$ denotes the right multiplication $b\mapsto ba$ for $b\in H$. The Cayley graph is called \emph{normal} if $R(H)$ is a normal subgroup of the automorphism group. 

Denote by $\Aut(H)$ the group of \emph{group automorphisms of $H$} and by $\Aut(H,T) = \{f \in \Aut(H) \, : \, \, f(T) = T \}$ the set of group automorphisms that setwise fix $T$. It is known~\cite{biggs1993algebraic} that $\Aut(H,T)$ is a subgroup of $\Aut(\Cay(H,T))$. The normalizer of $R(H)$ in $\Aut(\Cay(H,T))$ equals the semidirect product of the subgroups $R(H)$ and $\Aut(H,T)$, see~\cite{godsil1981full, Xu1998}. Hence, the Cayley graph $\Cay(H,T)$ is normal if and only if 
\[
\Aut(\Cay(H,T))=R(H)\rtimes \Aut(H,T).
\]
Cayley graphs that are normal can be interpreted as those that have the smallest possible automorphism groups. The identification of Cayley graphs that are normal is an open problem in the literature.

In this work, we consider the case where $H = S_n$, the symmetric group on $n$ elements, and $T$ is a set of transpositions generating $S_n$, i.e., permutations of the form $(i\,j)$ with $i,j \in [n] := \{1, \ldots, n\}$, $i \neq j$. Since for any transposition $(i\,j)$ we have $(i\,j)=(j\,i)$, we can identify a transposition with the unordered pair $\{i,j\}$. Let $E(T)$ denote the set of unordered pairs corresponding to the transpositions in $T$. Thus, the set of transpositions can be encoded as the edge set of a graph $G(T)=([n],E(T))$, the so-called \emph{transposition graph of $T$}. One easily verifies that the set $T$ generates $S_n$ if and only if $G(T)$ is connected, and hence $T$ is a minimal generating set for $S_n$ if and only if $G(T)$ is a tree~\cite{godsil2001algebraic}. Cayley graphs of the form $\Cay(S_n, T)$ are often studied as the topology of interconnection networks, see for instance~\cite{GanesanInterconnection,Heydemann,stacho1998bisection,zhang2005automorphism}. Moreover, Cayley graphs generated by transpositions have a close connection to several sorting algorithms~\cite{Jerrum} like bubble-sort and modified bubble-sort, since finding the cheapest way to sort a sequence of integers boils down to finding a shortest path in $\Cay(S_n,T)$. Finally, graphs of the form $\Cay(S_n,T)$ are recently exploited to find optimal embeddings of qubits in a quantum computing system~\cite{MatsuoYamashita,DeMeijerEtAl2024}. 

The automorphism group of graphs of the form $\Cay(S_n, T)$ has gathered notable attention in the literature. Godsil and Royle~\cite{godsil2001algebraic} show that if $G(T)$ is an asymmetric tree, then $\Aut(\Cay(S_n,T))$ is isomorphic to $S_n$. This result is strengthened by Feng~\cite{feng2006automorphism}, proving that $\Aut(\Cay(S_n,T))$ equals $R(S_n)\rtimes \Aut(S_n,T)$ when $G(T)$ is an arbitrary tree, implying the normality of $\Cay(S_n,T)$. Ganesan~\cite{ganesan2013automorphism} further strengthens this result, showing that the condition can be generalized to $G(T)$ having girth at least five.

There are known instances where $\Cay(S_n,T)$ is not normal. If $G(T)$ is a four-cycle, the group $\Aut(\Cay(S_n,T))$ has $768$ elements instead of $192$ (see \cite{ganesan2013automorphism}). If $G(T)$ is a complete graph with $n\geq 3$ vertices, then $\Aut(S_n, T)\cong (R(S_n)\rtimes \Inn(S_n))\rtimes \mathbb{Z}_2$ as was shown in~\cite{ganesan2015automorphism}. However, it was conjectured in~\cite{GanesanInterconnection} that these are the only two graph structures (with $n \geq 3$) for which normality does not hold. 

In the present paper, we prove that this conjecture is indeed true, providing a full answer to the question which Cayley graphs generated by transpositions are normal and which are not. Observe that for $n \leq 4$, all Cayley graphs of the form $\Cay(S_n,T)$ fall into one of the above-mentioned categories. Our goal is to prove the following theorem with respect to Cayley graphs with $n \geq 5$. 

\begin{theorem} \label{Thm:MainTheorem}
    Suppose that $n \geq 5$ and that $G(T)$ is not isomorphic to $K_n$. Then $\Aut(\Cay(S_n,T)) = R(S_n) \rtimes \Aut(S_n,T)$, implying that $\Cay(S_n,T)$ is normal.
\end{theorem}

It was shown in \cite{ganesan2016automorphism} that if $T\subseteq S_n$ is a generating set of transpositions such that $n\geq3$ and $\Cay(S_n,T)$ is normal, then $\Aut(\Cay(S_n,T))$ is the internal direct product of $R(S_n)$ and $L(\Aut(G(T)))$, where $L$ denotes the left regular representation. Hence, we obtain the following corollary.

\begin{corollary}
    \label{Thm:MainCorollary}
    Suppose that $n \geq 5$ and that $G(T)$ is not isomorphic to $K_n$. Then $\Aut(\Cay(S_n,T))$ is isomorphic to $S_n \times \Aut(G(T))$. 
\end{corollary} 

In Section~\ref{Section:Prelim} we review and derive some preliminary results on the automorphism group of Cayley graphs generated by transpositions. Section~\ref{Section:Proof} provides the proof of Theorem~\ref{Thm:MainTheorem}.

\section{Cayley graphs generated by transpositions} \label{Section:Prelim}
Throughout the rest of the paper, we let $T\subseteq S_n$ be a generating set of transpositions. We will denote by $G=G(T)$ the associated transposition graph and by $\Gamma=\Cay(S_n,T)$ the associated Cayley graph. We note that $\Gamma$ is bipartite since transpositions are odd permutations.

In this section we consider some structural properties of the graph $\Gamma$. We start by reviewing some preliminaries about $\Aut(\Gamma)$ in Section~\ref{Sec:PrelimAuto}, after which we derive some useful lemmas about transpositions and their induced structure in $\Gamma$ in Section~\ref{Sec:PrelimTrans}.

\subsection{Preliminaries on \boldmath $\Aut(\Gamma)$} \label{Sec:PrelimAuto}
We denote by $L(G)$ the line graph of $G$, by $K_n$ the complete graph on $n$ vertices and by $K_{n,m}$ the complete bipartite graph with partitions of size $n$ and $m$.

Every automorphism $\phi$ of $G$ induces an automorphism $\phi'$ of $L(G)$ given by $\phi'(\{i,j\})=\{\phi(i), \phi(j)\}$ for all $\{i,j\} \in E(T)$. It is easy to see that if $G$ has at most one isolated vertex and no component of two vertices, then the map $\phi\mapsto \phi'$ is an injective group homomorphism from $\Aut(G)$ to $\Aut(L(G))$. Whitney~\cite{Whitney1932} showed that, except for a few cases, this map is in fact an isomorphism. 

\begin{theorem}[\cite{Whitney1932}] \label{Theorem:Whitney}
Let $G$ be a graph with at most one isolated vertex, and no component equal to $K_2$. If $G$ has no component equal to $K_4$, $K_4-e$ (i.e., $K_4$ minus one edge) or a triangle with a pendant edge and $G$ does not have both a $K_3$-component and a $K_{3,1}$-component, then the map $\phi\mapsto \phi'$ is a group isomorphism from $\Aut(G)$ to $\Aut(L(G))$.
\end{theorem}
In particular, Theorem~\ref{Theorem:Whitney} implies that if $G$ is a connected graph on at least five vertices, every automorphism of $L(G)$ is induced by a unique automorphism of $G$. 



The proof of our main result, Theorem~\ref{Thm:MainTheorem}, relies on the exploitation of the following alternative characterization for normality of $\Gamma$ due to Ganesan~\cite{ganesan2015automorphism}. 

\begin{theorem}[\cite{ganesan2015automorphism}] \label{Theorem:CharacterizationNormal}
Let $T\subseteq S_n$ be a generating set of transpositions, where $n \geq 5$. Then $\Cay(S_n,T)$ is normal if and only if the identity map is the only automorphism of $\Cay(S_n,T)$ that fixes the identity vertex $\id$ and each of its neighbors. 
\end{theorem}

\subsection{Preliminaries on transpositions} \label{Sec:PrelimTrans}

We now derive a few preliminary results on the transpositions in $T$ and their induced structure in the Cayley graph $\Gamma$. 

\begin{lemma} \label{Lemma:Commute}
Let $a,b,c\in T$ be distinct and let $\sigma\in S_n$. Then
\begin{itemize}
\item[\textup{(a)}] Transpositions $a$ and $b$ commute if and only if $a$ and $b$ correspond to disjoint edges in $G$.
\item[\textup{(b)}] Transpositions $a$ and $b$ commute if and only if there is a unique $\tau\in S_n$ such that $(\sigma, a\sigma, \tau \sigma, b\sigma,\sigma)$ is a $4$-cycle in $\Gamma$. In this case, $\tau=ab$.
\item[\textup{(c)}] The edges corresponding to $a,b,c$ form a triangle in $G$ if and only if there exist $\tau_1,\tau_2\in S_n$ such that $\{\sigma, a\sigma, b\sigma, c\sigma, \tau_1\sigma ,\tau_2 \sigma\}$ induces a $K_{3,3}$ subgraph in $\Gamma$. In that case, $\{\tau_1,\tau_2\}=\{ab=ca=bc,ba=ac=cb\}$. 
\end{itemize}
\end{lemma}
Part (b) was also shown in \cite{ganesan2015automorphism}.
\begin{proof}
Part (a) is clear. For part (b) and (c) we may assume $\sigma=\id$ and set $p=ba^{-1}=ba$. The paths $(a,sa, tsa=b)$ of length $2$ from $a$ to $b$ in $\Gamma$ correspond bijectively to the decompositions $p=ts$ of $p$ as a product of transpositions $t,s\in T$. We have the following cases:
\begin{itemize}
\item[I.] Transpositions $a$ and $b$ commute. Say, without loss of generality, that $a=(1\,2)$ and $b=(3\,4)$. There are exactly two ways to write $p$ as a product of two transpositions: $p=(3\, 4)(1\,2)$ and $p=(1\, 2)(3\, 4)$. Hence, $(a, \id, b)$ and $(a, ba, b)$ are the only paths of length $2$ from $a$ to $b$.
We see that there is a unique $\tau$ such that $(\id, a, \tau, b,\id)$ is a $4$-cycle in $\Gamma$, and $\tau=ba=ab$. 

\item[II.] Transpositions $a$ and $b$ do not commute. Say, without loss of generality, that $a=(1\,2)$ and $b=(2\,3)$. There are exactly three ways to write $p=(1\,3\,2)$ as a product of two transpositions: 
\[
(1\,3\,2)=(2\,3)(1\,2)=(1\,3)(2\,3)=(1\,2)(1\,3).
\]
If $(1\, 3)\not\in T$ we see that $(a, \id, b)$ is the only path of length~$2$ from $a$ to~$b$. If $(1\, 3)\in T$, there are exactly three paths of length~$2$ from $a$ to~$b$: $(a,\id, b)$, $(a, \tau, b)$, $(a, \tau', b)$ where $\tau=(2\, 3)(1\, 2)=(1\, 3\, 2)$ and $\tau'=(1\, 3)(1\, 2)=(1\, 2\, 3)$. 
\end{itemize}
The proof of (b) now follows. If $a$ and $b$ commute, a unique $\tau$ as in (b) exists. If $a$ and $b$ do not commute there is either no such $\tau$ or there is more than one. 

To show (c), we first suppose that $a,b,c$ form a triangle in $G$. We may assume that $a=(1\, 2)$, $b=(2\, 3)$ and $c=(1\, 3)$. We see that $\{\id, (1\, 2\, 3 ),(1\, 3\, 2)\}\cup\{(1\, 2),(2\, 3),(1\, 3)\}$ induces a $K_{3,3}$ in $\Gamma$ and we can take $\tau_1=(1\, 2\, 3)$ and $\tau_2=(1\, 3\, 2)$.

Conversely, if $a,b$ are vertices of an induced $K_{3,3}$ in $\Gamma$, then there must be at least three paths of length $2$ from $a$ to $b$ (as $\Gamma$ is bipartite, $a$ and $b$ are in the same color class of the $K_{3,3}$.)  So $a$ and $b$ must be part of a triangle in $G$. 
\end{proof}
The following intermediate result is used to derive another substructure in $\Gamma$ based on non-commuting transpositions. 
\begin{lemma} \label{Lemma:Cases6cycle}
The $4$-tuples of transpositions $a,b,c,d$ such that $abcd=(1\,2\,3)$ and no two consecutive transpositions in the sequence $(1\, 2),a,b,c,d,(2\,3)$ commute, are precisely the tuples of the following eight types, where $k\not\in \{1,2,3\}$:
\begin{multicols}{2}
\begin{itemize}
    \item $(1\,3),(2\,3),(1\,2),(1\,3)$
    \item $(1\,3),(1\,k),(1\,2),(2\,k)$
    \item $(2\,3),(1\,2),(2\,3),(1\,2)$
    \item $(2\,3),(3\,k),(1\,k),(3\,k)$
    
\end{itemize}
\begin{itemize}
    \item $(2\,k),(2\,3),(3\,k),(1\,3)$
    \item $(2\,k),(1\,k),(3\,k),(2\,k)$
    \item $(1\,k),(3\,k),(1\,k),(1\,2)$
    \item $(1\,k),(1\,2),(2\,3),(3\,k)$

\end{itemize}
\end{multicols}

\end{lemma}
\begin{proof}
It is easy to check that in each of the eight cases $abcd=(1\, 2\, 3)$ and that consecutive transpositions in $(1\, 2),a,b,c,d,(2\, 3)$ do not commute. It remains to be shown that these are all possibilities.

For a transposition $t=(i\, j)$, we will say that $i$ and $j$ are the elements \emph{used by}~$t$. Any two consecutive transpositions in the sequence $a, b, c, d$ do not commute and must therefore use a common element. This implies that $a$, $b$, $c$ and $d$ together use at most five elements. Since $abcd = (1\,2\,3)$, three of these elements must be $1$, $2$ and~$3$. Hence, without loss of generality, we may assume that $a, b, c, d 
\in S_5$. 

Suppose that $a$, $b$, $c$, $d$ together use all elements of $\{1,\ldots, 5\}$. Since consecutive transpositions do not commute, the graph on vertex set $\{1,\ldots, 5\}$ and as edges the four pairs corresponding to $a,b,c,d$, is connected and therefore a tree. It now follows (see \cite{godsil2001algebraic}) that $abcd$ is a $5$-cycle, contradicting the fact that $abcd=(1\, 2\, 3)$. We conclude that either $4$ or $5$ is not used by $a,b,c,d$, so we may assume that $a,b,c,d\in S_4$. 

The statement now follows from checking all decompositions of $a(1\, 2\,3)d$ into a product of two transpositions for all sixteen combinations $a\in \{(1\, 3),(1\, 4),(2\, 3),(2\, 4)\}$ and $d\in \{(1\, 2),(1\, 3), (2\, 4),(3\, 4)\}$.
\end{proof}
Based on Lemma~\ref{Lemma:Cases6cycle}, we now show the following result, which is a generalization of Ganesan~\cite[Theorem 4]{ganesan2013automorphism}, which relied on the girth of $G$ to be at least 5.

\begin{lemma} \label{Lemma:CycleCayley} 
Let $\sigma\in S_n$ and let $s,t\in T$ be non-commuting. Suppose that $s$ and $t$ are not in a common cycle of length at most $4$ in $G$. Then there exist unique $\tau_1,\tau_2,\tau_3\in S_n$ such that 
\[
(\sigma, s\sigma, \tau_1 \sigma,\tau_2\sigma,\tau_3\sigma,t\sigma, \sigma)
\]
is a $6$-cycle in $\Gamma$ of which any two consecutive edges correspond to non-commuting transpositions. Moreover, we have 
\[
\tau_1=ts,\quad \tau_2=sts,\quad \tau_3=tsts=st.
\]
\end{lemma}
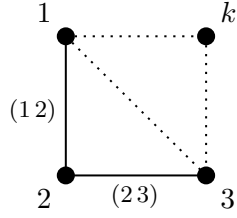
\begin{figure}
\begin{center}
\scalebox{1}{
\tikzset{every picture/.style={line width=0.75pt}} 

\begin{tikzpicture}[x=0.75pt,y=0.75pt,yscale=-1,xscale=1]

\draw  (115,65) -- (115,135) ;
\draw  (115,135) -- (185,135) ;
\draw  [dash pattern={on 0.84pt off 2.51pt}]  (115,65) -- (185,135) ;
\draw  [dash pattern={on 0.84pt off 2.51pt}]  (115,65) -- (185,65) ;
\draw  [dash pattern={on 0.84pt off 2.51pt}]  (185,65) -- (185,135) ;

\filldraw (115,65) circle (3pt);
\filldraw (115,135) circle (3pt);
\filldraw (185,65) circle (3pt);
\filldraw (185,135) circle (3pt);

\draw (115,65) node [anchor=south east][inner sep=5pt]    {$1$};
\draw (115,135) node [anchor=north east][inner sep=5pt]    {$2$};
\draw (185,135) node [anchor=north west][inner sep=5pt]    {$3$};
\draw (185,65) node [anchor=south west][inner sep=5pt]    {$k$};
\draw (85,95) node [anchor=north west][inner sep=0.75pt]  [font=\footnotesize]  {$(1\, 2)$};
\draw (136,138) node [anchor=north west][inner sep=0.75pt]  [font=\footnotesize]  {$(2\, 3)$};

\end{tikzpicture}
}
\end{center}
\caption{Subgraph of $G$ induced by vertices $1, 2, 3$ and any $k \notin \{1,2,3\}$. Existing and non-existing edges are denoted by solid and dotted lines, respectively. \label{Figure:CycleProof}}
\end{figure}
\begin{proof}
Without loss of generality, we assume that $s=(2\,3)$ and $t = (1\,2)$. Choosing $\tau_1 = (1\,2\,3)$, $\tau_2 = (1\, 3)$ and $\tau_3 = (1\, 3\, 2)$, it is clear that the given 6-cycle exists in $\Gamma$. 

To prove uniqueness, write 
\[
s=(2\, 3),\quad  \tau_1=d(2\, 3),\quad  \tau_2=cd(2\, 3),\quad  \tau_3=bcd(2\, 3),\quad  t=(1\, 2)=abcd(2\, 3). 
\]
 Observe that $a,b,c,d$ satisfy the conditions in Lemma~\ref{Lemma:Cases6cycle}. Since $(1\, 2)$ and $(2\, 3)$ are not part of a cycle of length at most $4$, we have $(1\, 3)\not\in T$ and for $k\geq 4$ at most one of $(1\, k)$ and $(3\, k)$ can be in $T$, see Figure~\ref{Figure:CycleProof}. Hence, from the eight types in Lemma~\ref{Lemma:Cases6cycle}, only one remains and we must have $a,b,c,d=(2\, 3), (1\, 2), (2\, 3), (1\, 2)$, proving the uniqueness of the induced $6$-cycle.
 \end{proof}

\section{Proof of main theorem} \label{Section:Proof}
We are now ready to prove Theorem~\ref{Thm:MainTheorem}. Let $\Phi \in \Aut(\Gamma)$ be an automorphism that fixes the identity vertex $\id$ and every transposition in $T$. Based on Theorem~\ref{Theorem:CharacterizationNormal}, it suffices to show that $\Phi$ fixes every vertex of $\Gamma$. Let 
\begin{align}
    U=\{\sigma\in S_n: \Phi(\tau)=\tau\text{ for all }\tau\in N(\sigma)\cup\{\sigma\}\},
\end{align} 
where $N(\sigma)$ denotes the set of neighbors of $\sigma$ in $\Gamma$. 
Observe that $\id \in U$, so $U$ is nonempty. 

Let $\sigma\in U$. Since $\Gamma$ is connected, it suffices to show that $a\sigma\in U$ for all $a\in T$. Observe that for every $a\in T$ the automorphism $\Phi$ fixes $a\sigma$ and  induces a bijection $N(a\sigma)\to N(a\sigma)$. Since each edge adjacent to $a\sigma$ is associated with a transposition in $T$, this induces a bijection $T \to T$. By part (b) of Lemma~\ref{Lemma:Commute}, it follows that this bijection preserves commutativity, i.e., commuting pairs of transpositions are mapped to commuting pairs of transpositions. So identifying transpositions with their corresponding edges in $G$, the bijection is an automorphism $\phi'_a$ of $L(G)$.  By Whitney's Theorem, i.e., Theorem~\ref{Theorem:Whitney}, this automorphism originates from a unique automorphism $\phi_a\in \Aut(G)$. 

It suffices to show that $\phi_a$ is the identity permutation for all $a\in T$. Before we do so, we need the following three intermediate results about the automorphisms~$\phi_a$. 

\begin{lemma}\label{Lemma:disjoint}
Let $a=(i\, j)\in T$. Then 
\begin{itemize}
\item[\textup{(i)}] $\phi_a(\{i,j\})=\{i,j\}$.
\item[\textup{(ii)}] Let $\{k,\ell\}$ be an edge of $G$ disjoint from $\{i,j\}$. Then $\phi_a(\{k,\ell\})=\{k,\ell\}$.
\end{itemize}
\end{lemma}
\begin{proof}
Since $a(a\sigma)=\sigma$ is fixed by $\Phi$ it follows that $\phi'_a$ fixes the edge $\{i,j\}$ and therefore $\phi_a$ fixes the set $\{i,j\}$.

Write $b=(k\,\ell)$. Since $\sigma, a\sigma, b\sigma$ are fixed by $\Phi$ and $a$ and $b$ commute, it follows by part (b) of Lemma~\ref{Lemma:Commute} that $ba\sigma$ is fixed by $\Phi$. This means that the set $\{k,\ell\}$ is fixed by $\phi_a$. 
\end{proof}

\begin{lemma} \label{Lemma:Triangle}
Let $\{i,j,k\}$ induce a triangle in $G$. Then $\phi_{(i\,j)}(k) = k$.
\end{lemma}
\begin{proof} 
    Since $\{i,j\}, \{j,k\}$ and $\{i,k\}$ form a triangle in $G$, it follows from part (c) of Lemma~\ref{Lemma:Commute} that the vertices $\sigma$, $(i\,j)\sigma$, $(j\,k)\sigma$ and $(i\,k)\sigma$ of $\Gamma$ are contained in a unique $K_{3,3}$ subgraph of $\Gamma$. Since $\sigma \in U$, it follows that these vertices are fixed by $\Phi$. This implies that the other two vertices in the $K_{3,3}$, i.e., $(j\,k)(i\,j)\sigma$ and $(i\,k)(i\,j)\sigma$, are setwise fixed by $\Phi$. In other words, $\phi'_{(i\,j)}$ fixes $\{\{j,k\}, \{i,k\}\}$. This implies that $\phi_{(i\,j)}(k) = k$. 
\end{proof}

\begin{lemma} \label{Lemma:InterTriangle}
    Let $\{i,j,k\}$ induce a triangle in $G$. Suppose that $\phi_{(i\,j)}$ fixes $i$, $j$ and~$k$. Then $\phi_{(i\,k)}$ and $\phi_{(j\,k)}$ also fix $i$, $j$ and~$k$. 
\end{lemma}
\begin{proof}
    That $\phi_{(i\, j)}$ fixes $i$, $j$ and $k$, implies that $\phi'_{(i\, j)}$ fixes the edges $\{i,k\}$ and $\{j,k\}$. So $\Phi$ fixes $\tau=(i\, k)(i\, j)\sigma$ and $\tau'=(j\, k)(i\, j)\sigma$. 
    
    Rewriting $\tau=(j\, k)(i\, k)\sigma$ and $\tau'=(i\, j)(i\, k)\sigma$ we see that $\phi'_{(i\, k)}$ fixes the edges $\{j,k\}$ and $\{i,j\}$, and also the edge $\{i,k\}$ by part (i) of Lemma~\ref{Lemma:disjoint}. It follows that $\phi_{(i\, k)}$ fixes $i$, $j$ and $k$. 
    Similarly, rewriting $\tau=(i\, j)(j\, k)\sigma$ and $\tau'=(i\, k)(j\, k)\sigma$, we find that $\phi'_{(j\, k)}$ fixes the edges $\{i,j\}$, $\{i,k\}$ and $\{j,k\}$, and therefore $\phi_{(j\, k)}$ fixes $i$, $j$ and $k$.     
\end{proof}

\begin{lemma} \label{Lemma:Cycle}
    Let $\{i,j\}$ and $\{i,k\}$ be edges of $G$ that are not on a common cycle of length at most $4$. Then $\phi_{(i\,j)}$ fixes $i$, $j$ and $k$.
\end{lemma}
\begin{proof}
    As the transpositions $(i\,j)$ and $(i\,k)$ do not commute and their corresponding edges are not in a common cycle of length at most $4$, it follows from Lemma~\ref{Lemma:CycleCayley} that the vertices $(i\, j)\sigma$, $\sigma$ and $(i\,k)\sigma$ are consecutive vertices on a unique $6$-cycle in $\Gamma$ with the property that any two consecutive edges correspond to transpositions that do not commute. Since $\Phi$ fixes the vertices $(i\, j)\sigma$, $\sigma$ and $(i\,k)\sigma$ it must fix all vertices in this $6$-cycle. In particular, it fixes $(i\,k)(i\,j)\sigma$, so $\phi'_{(i\,j)}$ fixes $\{i,k\}$. Since $\phi'_{(i\,j)}$ also fixes $\{i,j\}$, it follows that $\phi_{(i\,j)}$ fixes $i$, $j$ and~$k$. 
\end{proof}

Now, we are ready to prove that $\phi_a$ is the identity permutation for all $a \in T$. Without loss of generality, we will assume $a=(1\,2)$. Let $H:= G[\{3,\ldots, n\}]$ be the graph obtained by deleting the vertices $1$ and $2$ from $G$. Let $C$ be a connected component of $H$. We consider the following cases. 
\begin{itemize}
\item[I.] Suppose $C$ consists of a single vertex $v$. Then $\phi_a$ fixes $1$, $2$ and $v$.

\begin{proof}
First suppose that $\{1,2,v\}$ induces a triangle in $G$. Since $n \geq 5$, it follows that at least one of the vertices $1$ and $2$ has a neighbor in a different component. So without loss of generality, we will assume $d(1) > d(v)$, where $d(\cdot)$ denotes the degree of a vertex in $G$. Since the automorphism $\phi_{(1\,v)}$ fixes the pair $\{1,v\}$ and these vertices have different degrees, it must fix both vertex $1$ and $v$. By Lemma~\ref{Lemma:Triangle}, it also fixes vertex $2$. Hence, by Lemma~\ref{Lemma:InterTriangle}, it follows that also $\phi_a$ fixes vertices $1$, $2$ and $v$.

Next, suppose that $\{1, 2, v\}$ does not induce a triangle in $G$. Without loss of generality, suppose $\{2,v\}\not\in E$, but $\{1,v\} \in E$. Since $\{1,v\}$ and $\{1,2\}$ are adjacent edges not in a cycle in $G$, Lemma~\ref{Lemma:Cycle} implies that $\phi_a$ fixes the vertices~$1$, $2$ and~$v$.
\end{proof}

\item[II.] Suppose that $C$ has $2$ vertices, say $v$ and $w$. We consider the following subcases, see Figure~\ref{Figure:caseII}.
\begin{itemize}
\item[(a)] Suppose that $\{1,2,v,w\}$ induces a path or a $K_4-e$. Then $\phi_a$ fixes $1$, $2$, $v$ and $w$.

\begin{proof}
Since $\phi_a$ fixes the pair $\{v,w\}$ by Lemma~\ref{Lemma:disjoint} and $d(v)\neq d(w)$, it fixes both $v$ and~$w$. Since $\phi_a$ also fixes the pair $\{1,2\}$ and they do not have the same number of neighbors among $\{v,w\}$, $\phi_a$ must fix both $1$ and $2$.
\end{proof}

\item[(b)] Suppose that $\{1,2,v,w\}$ induces a triangle with a pendant edge. Then $\phi_a$ fixes $1$, $2$, $v$ and $w$.

\begin{proof}
If $\{1,2\}$ is the pendant edge, we may assume without loss of generality that the triangle is induced by $\{2,v,w\}$. Then $\{1,2\}$ and $\{2,w\}$ are adjacent edges not in a common cycle, so by Lemma~\ref{Lemma:Cycle} the vertices $1$, $2$ and $w$ are fixed by $\phi_a$. Then since $\{v,w\}$ is fixed by Lemma~\ref{Lemma:disjoint}, also $v$ is fixed by $\phi_a$. 

If $\{v,w\}$ is the pendant edge, we may assume without loss of generality that $\{1,2,v\}$ induces a triangle. We first consider the automorphism $\phi_{(1\, v)}$. Since $\{1,v\}$ and $\{v,w\}$ are adjacent edges not in a common cycle, Lemma~\ref{Lemma:Cycle} implies that $\phi_{(1\, v)}$ fixes the vertices $1$, $v$ and $w$. By Lemma~\ref{Lemma:Triangle}, $\phi_{(1\, v)}$ also fixes vertex $2$. It now follows from Lemma~\ref{Lemma:InterTriangle} that also $\phi_a$ fixes vertices $1$, $2$ and $v$, and therefore also vertex $w$ as it fixes $\{v,w\}$.
\end{proof}

\item[(c)] Suppose that $\{1,2,v,w\}$ induces a $K_4$. Then $\phi_a$ fixes $1$, $2$, $v$ and $w$.

\begin{proof}
Since $G$ has at least five vertices, we may assume without loss of generality that $d(1)\geq 4$ in the graph $G$. Since $d(v)=3<d(1)$, the automorphism $\phi_{(1\, v)}$ must fix vertices $1$ and $v$. By Lemma~\ref{Lemma:Triangle}, also vertex $2$ is fixed by $\phi_{(1\, v)}$. By Lemma~\ref{Lemma:InterTriangle} it now follows that also $\phi_a$ fixes vertices $1$, $2$ and $v$, and therefore also vertex $w$.
\end{proof}

\item[(d)] Suppose that $\{1,2,v,w\}$ induces a $C_4$. Then either $\phi_a$ fixes the four vertices $1$, $2$, $v$, $w$, or $\phi_a$ swaps vertices $1$ and $2$ and also swaps vertices $v$ and $w$.

\begin{proof}
Since $\phi_a$ fixes the pair $\{1,2\}$ and the pair $\{v,w\}$ by Lemma~\ref{Lemma:disjoint}, these are clearly the only two options.
\end{proof}
\end{itemize}

\begin{figure}[h]
\centering
\scalebox{0.7}{
\tikzset{every picture/.style={line width=0.75pt}}

\begin{tikzpicture}[x=0.75pt,y=0.75pt,yscale=-1,xscale=1]

\draw    (45.5,265.5) -- (45.5,335.5);
\draw    (45.5,335.5) -- (115.5,335.5);
\draw    (45.5,265.5) -- (115.5,265.5);

\filldraw (45.5,265.5) circle (3pt);
\filldraw (45.5,335.5) circle (3pt);
\filldraw (115.5,335.5) circle (3pt);
\filldraw (115.5,265.5) circle (3pt);

\draw (45.5,265.5) node [anchor=south east][inner sep=4pt]  {$1$};
\draw (115.5,265.5) node [anchor=south west][inner sep=4pt]  {$2$};
\draw (45.5,335.5) node [anchor=north east][inner sep=4pt]  {$v$};
\draw (115.5,335.5) node [anchor=north west][inner sep=4pt] {$w$};

\begin{scope}[yshift=120pt]
\draw    (45.5,265.5) -- (115.5,265.5);
\draw    (115.5,265.5) -- (45.5,335.5);
\draw    (45.5,335.5) -- (115.5,335.5);

\filldraw (45.5,265.5) circle (3pt);
\filldraw (45.5,335.5) circle (3pt);
\filldraw (115.5,335.5) circle (3pt);
\filldraw (115.5,265.5) circle (3pt);

\draw (45.5,265.5) node [anchor=south east][inner sep=4pt]  {$1$};
\draw (115.5,265.5) node [anchor=south west][inner sep=4pt]  {$2$};
\draw (45.5,335.5) node [anchor=north east][inner sep=4pt]  {$v$};
\draw (115.5,335.5) node [anchor=north west][inner sep=4pt] {$w$};
\end{scope}

\begin{scope}[xshift=105pt]
\draw    (45.5,265.5) -- (45.5,335.5);
\draw    (45.5,265.5) -- (115.5,265.5);
\draw    (45.5,265.5) -- (115.5,335.5);

\draw    (45.5,335.5) -- (115.5,335.5);
\draw    (115.5,265.5) -- (115.5,335.5);

\filldraw (45.5,265.5) circle (3pt);
\filldraw (45.5,335.5) circle (3pt);
\filldraw (115.5,335.5) circle (3pt);
\filldraw (115.5,265.5) circle (3pt);

\draw (45.5,265.5) node [anchor=south east][inner sep=4pt]  {$1$};
\draw (115.5,265.5) node [anchor=south west][inner sep=4pt]  {$2$};
\draw (45.5,335.5) node [anchor=north east][inner sep=4pt]  {$v$};
\draw (115.5,335.5) node [anchor=north west][inner sep=4pt] {$w$};
\end{scope}

\begin{scope}[xshift=105pt, yshift=120]
\draw    (45.5,265.5) -- (45.5,335.5);
\draw    (45.5,265.5) -- (115.5,265.5);
\draw    (115.5,265.5) -- (45.5,335.5);

\draw    (45.5,335.5) -- (115.5,335.5);
\draw    (115.5,265.5) -- (115.5,335.5);

\filldraw (45.5,265.5) circle (3pt);
\filldraw (45.5,335.5) circle (3pt);
\filldraw (115.5,335.5) circle (3pt);
\filldraw (115.5,265.5) circle (3pt);

\draw (45.5,265.5) node [anchor=south east][inner sep=4pt]  {$1$};
\draw (115.5,265.5) node [anchor=south west][inner sep=4pt]  {$2$};
\draw (45.5,335.5) node [anchor=north east][inner sep=4pt]  {$v$};
\draw (115.5,335.5) node [anchor=north west][inner sep=4pt] {$w$};
\end{scope}

\begin{scope}[xshift=210pt]
\draw    (45.5,265.5) -- (115.5,265.5);
\draw    (45.5,335.5) -- (115.5,335.5);
\draw    (115.5,265.5) -- (115.5,335.5);
\draw    (115.5,265.5) -- (45.5,335.5);

\filldraw (45.5,265.5) circle (3pt);
\filldraw (45.5,335.5) circle (3pt);
\filldraw (115.5,335.5) circle (3pt);
\filldraw (115.5,265.5) circle (3pt);

\draw (45.5,265.5) node [anchor=south east][inner sep=4pt]  {$1$};
\draw (115.5,265.5) node [anchor=south west][inner sep=4pt]  {$2$};
\draw (45.5,335.5) node [anchor=north east][inner sep=4pt]  {$v$};
\draw (115.5,335.5) node [anchor=north west][inner sep=4pt] {$w$};
\end{scope}

\begin{scope}[xshift=265pt, yshift=120]
\draw    (45.5,265.5) -- (45.5,335.5);
\draw    (45.5,265.5) -- (115.5,265.5);
\draw    (115.5,265.5) -- (45.5,335.5);
\draw    (45.5,335.5) -- (115.5,335.5);

\filldraw (45.5,265.5) circle (3pt);
\filldraw (45.5,335.5) circle (3pt);
\filldraw (115.5,335.5) circle (3pt);
\filldraw (115.5,265.5) circle (3pt);

\draw (45.5,265.5) node [anchor=south east][inner sep=4pt]  {$1$};
\draw (115.5,265.5) node [anchor=south west][inner sep=4pt]  {$2$};
\draw (45.5,335.5) node [anchor=north east][inner sep=4pt]  {$v$};
\draw (115.5,335.5) node [anchor=north west][inner sep=4pt] {$w$};
\end{scope}

\begin{scope}[xshift=315pt]
\draw    (45.5,265.5) -- (115.5,265.5);
\draw    (45.5,335.5) -- (115.5,335.5);
\draw    (45.5,265.5) -- (115.5,335.5);
\draw    (45.5,265.5) -- (45.5,335.5);

\filldraw (45.5,265.5) circle (3pt);
\filldraw (45.5,335.5) circle (3pt);
\filldraw (115.5,335.5) circle (3pt);
\filldraw (115.5,265.5) circle (3pt);

\draw (45.5,265.5) node [anchor=south east][inner sep=4pt]  {$1$};
\draw (115.5,265.5) node [anchor=south west][inner sep=4pt]  {$2$};
\draw (45.5,335.5) node [anchor=north east][inner sep=4pt]  {$v$};
\draw (115.5,335.5) node [anchor=north west][inner sep=4pt] {$w$};
\end{scope}




\begin{scope}[xshift=420pt]
\draw    (45.5,265.5) -- (115.5,265.5);
\draw    (45.5,265.5) -- (45.5,335.5);
\draw    (45.5,265.5) -- (115.5,335.5);
\draw    (45.5,335.5) -- (115.5,335.5);
\draw    (45.5,335.5) -- (115.5,265.5);
\draw    (115.5,265.5) -- (115.5,335.5);

\filldraw (45.5,265.5) circle (3pt);
\filldraw (45.5,335.5) circle (3pt);
\filldraw (115.5,335.5) circle (3pt);
\filldraw (115.5,265.5) circle (3pt);

\draw (45.5,265.5) node [anchor=south east][inner sep=4pt]  {$1$};
\draw (115.5,265.5) node [anchor=south west][inner sep=4pt]  {$2$};
\draw (45.5,335.5) node [anchor=north east][inner sep=4pt]  {$v$};
\draw (115.5,335.5) node [anchor=north west][inner sep=4pt] {$w$};
\end{scope}

\begin{scope}[xshift=420pt, yshift=120]
\draw    (45.5,265.5) -- (45.5,335.5);
\draw    (45.5,265.5) -- (115.5,265.5);
\draw    (115.5,265.5) -- (115.5,335.5);
\draw    (45.5,335.5) -- (115.5,335.5);

\filldraw (45.5,265.5) circle (3pt);
\filldraw (45.5,335.5) circle (3pt);
\filldraw (115.5,335.5) circle (3pt);
\filldraw (115.5,265.5) circle (3pt);

\draw (45.5,265.5) node [anchor=south east][inner sep=4pt]  {$1$};
\draw (115.5,265.5) node [anchor=south west][inner sep=4pt]  {$2$};
\draw (45.5,335.5) node [anchor=north east][inner sep=4pt]  {$v$};
\draw (115.5,335.5) node [anchor=north west][inner sep=4pt] {$w$};
\end{scope}

\draw   (10,220) -- (570+140,220) -- (570+140,550) -- (10,550) -- cycle ;
\draw    (150,220) -- (150,550) ;
\draw    (290,220) -- (290,550) ;
\draw    (430+140,220) -- (430+140,550) ;
\draw    (430+140,380) -- (570+140,380) ;

\draw (173,520) node [anchor=north west][align=left] {(a) $\displaystyle K_{4} \ -\ e$ };
\draw (323,520) node [anchor=north west][align=left] {(b) triangle with pendant edge};
\draw (488.5+123,350) node [anchor=north west][align=left] {(c) $\displaystyle K_{4}$ };
\draw (492+123,520) node [anchor=north west][align=left] {(d) $\displaystyle C_{4}$ };
\draw (48,520) node [anchor=north west][align=left] {(a) path};

\end{tikzpicture}
}
\caption{Options (up to renaming $v$ and $w$) for the subgraph of $G$ induced by $\{1, 2, v, w\} $ corresponding to case II. \label{Figure:caseII}}
\end{figure}
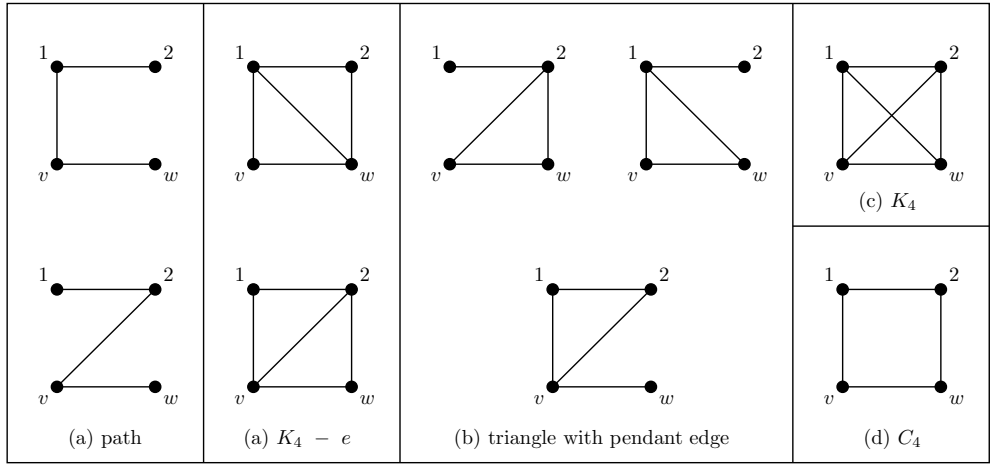

\item[III.] Suppose that $C$ has at least three vertices. Then vertices $1$ and $2$ and all vertices in $C$ are fixed by $\phi_a$.

\begin{proof}
Since all edges in $C$ are disjoint from $\{1,2\}$, it follows from Lemma~\ref{Lemma:disjoint} that $\phi_a$ fixes every edge of $C$. Since $C$ is connected and has at least three vertices, there must be a vertex $v$ that is incident with multiple edges. Hence $v$ is fixed by $\phi_a$. Since $\phi_a$ fixes all edges, any neighbor of a fixed vertex is fixed, so $\phi_a$ fixes every vertex of $C$.  

If vertex $1$ and $2$ do not have the same neighbors in $C$, then $1$ and $2$ are fixed by $\phi_a$ and we are done. Therefore, we will assume that vertex $1$ and $2$ have the same set $S$ of neighbors in $C$. Suppose  there exists a $v \in S$ with $d(v) \neq d(1)$ or $d(v)\neq d(2)$, say without loss of generality that $d(v)\neq d(1)$. Then vertex $1$ and $v$ are fixed in $\phi_{(1\,v)}$, and by Lemma~\ref{Lemma:Triangle} so is vertex $2$. It then follows by Lemma~\ref{Lemma:InterTriangle} that $1$,  $2$ and $v$ are also fixed by $\phi_a$ and we are done. So from now on, we will assume that $d(1) = d(2) = d(v)$ for all $v \in S$. 

If $C \setminus S = \emptyset$, then $d(1)=d(2)\geq |S|+1\geq d(v)=d(1)$ for all $v\in S$. But equality can only hold if every two vertices in $S$ are connected by an edge and $1$ and $2$ have no neighbors in other components. This means that $G=K_{2+|S|}$, which is prevented by the theorem statement. 

So we can assume that $C \setminus S \neq \emptyset$. Since $C$ is connected, there must exist a $w \in S$ that is adjacent to a vertex $u\in C \setminus S$. This situation is depicted in Figure~\ref{Figure:C4proof}. 
If $u$ has only one neighbor in $S$, then $\{1,w\}$ and $\{w,u\}$ are not in a common cycle of length at most $4$. Hence, by Lemma~\ref{Lemma:Cycle}, the automorphism $\phi_{(1\, w)}$ fixes vertices $1$, $w$ and $u$. It also fixes vertex $2$ by lemma~\ref{Lemma:Triangle}. It follows by Lemma~\ref{Lemma:InterTriangle} that also $\phi_a$ fixes vertices $1$, $2$ and $w$ and we are done.

So we may assume that every $u\in C\setminus S$ that is a neighbor of $w$ has at least one other neighbor in $S$. Set $W=V(C)\cup\{2\}\setminus\{w\}$. It follows that $G[W]$ is connected and has at least three vertices. So $\phi_{(1\, w)}$ fixes all vertices in $W$. Since $1$ and $w$ do not have the same neighbors in $W$, also $1$ and $w$ are fixed by $\phi_{(1\, w)}$. By Lemma~\ref{Lemma:InterTriangle} it now follows that also $\phi_a$ fixes vertices $1$, $2$ and $w$. 
\end{proof}

\begin{figure}[h]
\begin{center}
\scalebox{0.7}{
\tikzset{every picture/.style={thick}} 

\begin{tikzpicture}[]

\draw (0,1)--(0,3);
\draw (0,1)--(4,0)--(0,3);
\draw (0,1)--(4,1)--(0,3);
\draw (0,1)--(4,2);
\draw[red] (4,2)--(0,3);
\draw (0,1)--(4,3)--(0,3);
\draw (0,1)--(4,4)--(0,3);

\draw[red] (4,2)--(8,2);
\draw[dashed] (8,2)--(4,0);
\draw[dashed] (8,2)--(4,1);
\draw[dashed] (8,2)--(4,3);
\draw[dashed] (8,2)--(4,4);

\draw (4,2) ellipse (1cm and 2.8cm);
\draw (8,2) ellipse (1cm and 2.8cm);

\draw (0,3) node [anchor=south east][inner sep=4pt]   {$1$};
\draw (0,1) node [anchor=north east][inner sep=4pt]   {$2$};
\draw (4,5) node [anchor=south][inner sep=0.75pt]   {$S$};
\draw (8,5) node [anchor=south][inner sep=0.75pt]   {$C\setminus S$};
\draw (4,2) node [anchor=north west][inner sep=4pt]   {$w$};
\draw (8,2) node [anchor=north west][inner sep=4pt]   {$u$};

\filldraw [black] (0,1) circle (3pt);
\filldraw [black] (0,3) circle (3pt);

\filldraw [black] (4,0) circle (3pt);
\filldraw [black] (4,1) circle (3pt);
\filldraw [black] (4,2) circle (3pt);
\filldraw [black] (4,3) circle (3pt);
\filldraw [black] (4,4) circle (3pt);

\filldraw [black] (8,0) circle (3pt);
\filldraw [black] (8,1) circle (3pt);
\filldraw [black] (8,2) circle (3pt);
\filldraw [black] (8,3) circle (3pt);
\filldraw [black] (8,4) circle (3pt);

\end{tikzpicture}
}
\caption{Overview of component $C$ with $|C| \geq 3$ when $S \setminus C \neq \emptyset$. Edge $\{w,u\}$ does exist. Dotted edges might or might not exist. 
\label{Figure:C4proof}}
\end{center}
\end{figure}
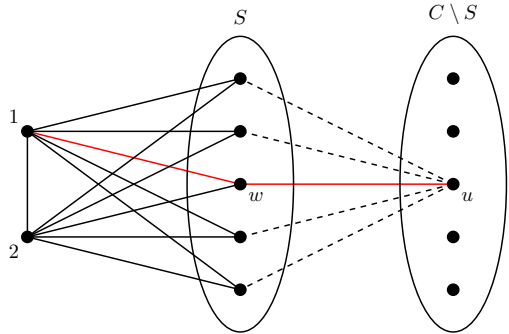

\end{itemize}

Combining the three cases above, we see that $\phi_a$ is the identity unless every component of $H$ is of type II(d). In that case, there must be $t\geq 2$ such components and $\{1,2\}$ is the unique edge of $G$ that is adjacent to $2t$ edges. Thus, the transposition $(1\,2)$ is the unique transposition that does not commute with $2t$ other transpositions. So the edges of the Cayley graph $\Gamma$ corresponding to $(1\,2)$ are permuted among themselves by $\Phi$. Since there is a unique $6$-cycle through vertices $(1\,2)\sigma, \sigma, (1\,v)\sigma$ in $\Gamma$ using three edges corresponding to $(1\,2)$, the vertices of this $6$-cycle must be fixed by $\Phi$, implying that $\phi_{a}$ fixes $1$, $2$ and $v$. 

We conclude that the automorphism $\phi_{a}$ is the identity permutation, completing the proof of Theorem~\ref{Thm:MainTheorem}.

\bibliographystyle{plain}
\bibliography{cayley}

\end{document}